\documentclass[12pt]{amsart}

\usepackage{amsfonts}
\usepackage{graphicx}
\usepackage{tabularx}
\usepackage{array}
\usepackage[usenames,dvipsnames]{color}
\usepackage{comment}
\usepackage{amsmath}
\usepackage{amsthm}
\usepackage{amssymb}
\usepackage{fullpage}
\usepackage[dvipsnames]{xcolor}
\usepackage{tikz}
\usepackage{listings}
\usepackage{dsfont}

\newtheorem{theorem}{Theorem}[section]

\newtheorem{conjecture}[theorem]{Conjecture}

\newtheorem{lemma}[theorem]{Lemma}
\newtheorem{corollary}[theorem]{Corollary}
\theoremstyle{definition}

\def\epsilon{\varepsilon}
\def\dist{\text{dist}}

\title{Graphs with Large Girth and Small Cop Number}
\author{Alexander Clow}
\address{Department of Mathematics, Simon Fraser University}
\email{alexander\_clow@sfu.ca}
\date{\today}

\begin{document}

\begin{abstract}
In this paper we consider the cop number of graphs with no, or few, short cycles. We show that when $G$ is graph of girth $g$ and the minimum degree $\delta \geq 2$, then $c(G) = O(n\log(n)(\delta-1)^{-\lfloor \frac{g+1}{4} \rfloor})$ as a function of $n$. This extends work of Frankl and implies that if $G$ is large and dense in the sense that $\delta \geq n^{\frac{2}{g}+\epsilon}$, then $G$ satisfies Meyniel's conjecture, that is $c(G) = O(\sqrt{n})$. Moreover, it implies that if $G$ is large and dense in the sense that there $\delta \geq n^{\epsilon}$, some $\epsilon >0$, while also having girth $g \geq 7$, then there exists an $\alpha>0$ such that $c(G) = O(n^{1-\alpha})$, thereby satisfying the weak Meyniel's conjecture. Of course, this implies similar results for dense graphs with small, that is $O(n^{1-\alpha})$, numbers of short cycles, as each cycle can be broken by adding a single cop. 
\end{abstract}

\maketitle


\section{Introduction}

In this paper we consider the game of \emph{cops and robbers} on connected graphs. As a result we suppose all graphs in this paper are connected. Cops and robbers is a 2-player game played on a graph $G = (V,E)$. To begin the game, the cop player places $k$ cops onto vertices of the graph, then the robber player chooses a vertex of the graph to place the robber. From here the cop and robber players alternate turns. On a player's turn, they can (but do not have to) move each their pieces to a vertex adjacent to its position.
Note this allows a player to not move a given cop or robber, should they desire to do so. The cops win if after finitely many moves a cop is able to move onto the vertex currently occupied by the robber, which we call capturing the robber, while the robber wins if there exists a strategy which indefinitely evades capture. Here both players are assumed to play optimally. The least number of cops required to win the game no matter which vertex the robber begins on is known as the cop number of a graph, denoted $c(G)$.

Cops and robbers was introduced independently by Quilliot \cite{quilliot1983problemes}, and Nowakowski and Winkler \cite{nowakowski1983vertex}; while the cop number of a graph was first introduced by Aigner and Fromme \cite{aigner1984game}. In recent decades cops and robbers has seen a remarkable amount of attention with many general results as well as results for specific graph classes appearing in the literature. For instance, the cop number of planar graphs \cite{aigner1984game}, graphs of higher genus \cite{bonato2020topological, bowler2019bounding, quilliot1985short, schroder2001copnumber}, Cayley graphs \cite{bradshaw2020proof,frankl1987cops}, graph products \cite{neufeld1998game}, and random graphs \cite{bollobas2013cops,luczak2010chasing, pralat2016meyniel} have all been extensively studied. Significant attention has also been paid to computational questions involving cop number, with MacGillivray and Clarke \cite{clarke2012characterizations} showing that deciding if a graph is $k$-cop win is fixed parameter polynomial time in the order of the graph and $k$, while Kinnersley \cite{kinnersley2015cops} proved that determining if $c(G) \leq k$ is EXPTIME-complete when $k$ is not fixed.

The most famous problem in cops and robbers is undoubtedly \emph{Meyniel's conjecture}, which states that for all graphs $G$, $c(G) = O(\sqrt{n})$. If correct, then this bound could not be improved as there are known graph families with $c(G) = \Omega(\sqrt{n})$ \cite{bonato2013cops, pralat2010does}. Meyniel's conjecture is resolved for a number of graph classes such as abelian Cayley graphs \cite{bradshaw2020proof} and random graphs \cite{bollobas2013cops, pralat2016meyniel}, but the larger conjecture remains widely open despite significant effort \cite{chiniforooshan2008better,frieze2012variations, lu2012meyniel, scott2011bound}. In fact, it remains to be shown that there exists a $\alpha>0$ such that $c(G) = O(n^{1-\alpha})$, with this problem sometimes being dubbed the \emph{weak Meyniel's conjecture}. For more on Meyniel's conjecture see \cite{baird2013meyniel}.

In this paper we focus on graphs with no, or few short cycles. The \emph{girth} of a graph $G$ is the length of the shortest cycle in $G$.  Considering graphs of large girth allows our analysis to avoid complicated local structures appearing in many graphs. Cops and robbers on graphs with large girth was first considered by Frankl in \cite{frankl1987cops} who showed that that $c(G) > (\delta-1)^{(1-o(1))\frac{g}{8}}$ and if $G$ is a $d$-regular graph of girth $g$, then $c(G) \leq n (1+(\frac{g}{4} +o(1)) \ln{d}) d^{-(1-o(1))\frac{g}{4}}$. We extend Frankl's upper bound for the cop number of $d$-regular graphs with large girth to an upper bound for all graphs with large girth and sufficiently large minimum degree, see Theorem~\ref{Thm: Cop Number Upper Bound}.

\begin{theorem}\label{Thm: Cop Number Upper Bound}
Let $G$ be an $n$ vertex graph of girth $g$ and minimum degree $\delta\geq 2$, then 
$$
c(G) \leq \frac{6n\log(n)}{(\delta-1)^{\lfloor\frac{g+1}{4}\rfloor}}
$$
\end{theorem}

\begin{corollary}\label{Corollary: Meyniel's Conjecutre Implication}
If $G = (V,E)$ has girth $g$ and $\delta = \delta(G) \geq n^{x}$ where $x$ is a constant, then 
$$
c(G) = O(n^{1-\frac{xg}{4}+o(1)}).
$$ 
Thus, we have shown that all large graphs with girth $g$ and $\delta \geq n^{\frac{2}{g}+\epsilon}$ must have $c = o(\sqrt{n})$.
\end{corollary}

 Therefore, all such graphs satisfy Meyniel's conjecture. Furthermore, Corollary~\ref{Corollary: Meyniel's Conjecutre Implication} implies that if $G$ is a sufficiently large graph with girth at least $g$ and $\delta \geq n^{\epsilon}$ where $\epsilon>0$, then there exists an $\alpha>0$ such that $c(G) = o(n^{1-\alpha})$ thereby showing that all such graphs satisfy the weak Meyniel's conjecture. 

We note that this implies that if $G$ is a graph with at most $O(n^{1-\alpha})$ cycles of length $g-1$ or less, we may conclude that if $\delta \geq n^{\frac{4\alpha}{g}+\epsilon}$, then $c(G) = O(n^{1-\alpha})$. This is because choosing one vertex per cycle to form a set of vertices $S$, then placing a cop at each vertex in $S$, means that 
\[
c(G) \leq |S|+c(G-S) = O(n^{1-\alpha})
\]
given the cops on $S$ can remain on $S$ for the full game, thereby confining the robber to $G-S$ where $c(G-S)$ cops will be able to capture the robber.

The most common method for constructing a graph with cop number at least $k$ is given by Aigner and Fromme \cite{aigner1984game} who showed that a graph $G$ with girth at least $5$ and minimum degree $k$ has $c(G) \geq k$. Frankl's \cite{frankl1987cops} lower bound on cop number in terms of girth $g$ and minimum degree $d$ is a generalisation of this. It is natural then to ask what is the best constant $r$ such that there exists a graph $G$ with minimum degree $d$ and girth $g$ such that $c(G) \leq d^{rg}$. Given Theorem~\ref{Thm: Cop Number Upper Bound} is an upper bound in terms of $n$, $d$, and $g$, it is natural to consider graphs $G$ with minimum degree $d$ and girth $g$ whose order is as small as possible. That is $(d,g)$-cages.

A $d$-regular graph with girth $g$ of minimal order is called a $(d,g)$-cage and the order of a $(d,g)$-cage is denoted $n(d,g)$. The problem of constructing $(d,g)$-cages and determining $n(d,g)$ has been extensively studied by a number of authors. For a survey on such problems see \cite{exoo2012dynamic}. Of particular interest for us is the following bound by Lazebnik, Ustimenko, and Woldar \cite{lazebnik1997upper}; let $d \geq 2$ and $g \geq 5$ be integers, and let $q$ denote an odd prime power for which $d \leq q$. Then 
\[
n(d,g) \leq 2dq^{\frac{3}{4}g-a}
\]
where $a = 4,11/\! 4, 7/ 2, 13/ 4$, for $g \equiv 0,1,2,3 \mod{4}$ respectively. Combining this result with Theorem~\ref{Thm: Cop Number Upper Bound} gives the following upper bound on the cop number of $(p^k,4t-1)$-cages. Of course a similar result can be given for $(d,g)$-cages, however we limit ourselves to the $d=p^k$, $g=4t-1$ case to simplify the resulting equations.

\begin{corollary}\label{Corollary: Cop Number of Specific Cages}
If $G = (V,E)$ is a $(p^k,4t-1)$-cage where $t\geq 1$ is fixed and $p$ is an odd prime, then $c(G) = O(p^{2tk}) = O(\delta^{(1-o(1))\frac{g}{2}})$ as $k \rightarrow \infty$.
\end{corollary}

So letting $r$ be the least constant such that there exists a graph $G$ with minimum degree $d$ and girth $g$ (both large) with $c(G) \leq d^{rg}$, Corollary~\ref{Corollary: Cop Number of Specific Cages} tells us that $r \leq 1/2$. However, this remains far from Frankl's lower bound that $r \geq 1/8$. Luckily, the lower bound for cop number in terms of minimum degree and girth given by Frankl \cite{frankl1987cops} was recently improved by Bradshaw, Hosseini, Mohar, and Stacho in \cite{bradshaw2023cop} who showed that $c(G) \geq \frac{1}{g} (\delta - 1)^{\lfloor \frac{g-1}{4}\rfloor}$. Furthermore, it was conjectured in \cite{bradshaw2023cop} that the exponential coefficient of $\frac{1}{4}$ cannot be improved, see Conjecture~\ref{Conj: Large Girth Small Cop Number}. This conjecture is quite reasonable as it is implied by combining Meyniel's conjecture and a conjecture that exists in folklore stating that there exist $d$-regular graphs of girth $g$ and order $(d-1)^{(1+o(1))\frac{g}{2}}$. 

\begin{conjecture}\label{Conj: Large Girth Small Cop Number}
There exist graphs $G$ with $\delta(G) = \delta$ and girth $g$ where 
\[
c(G) \leq (\delta-1)^{(1+o(1))\frac{g}{4}}.
\]
\end{conjecture}

We note that Theorem~\ref{Thm: Cop Number Upper Bound} removes the need for Meyniel's conjecture to show Conjecture~\ref{Conj: Large Girth Small Cop Number}. That is, if there exist $d$-regular graphs $G$ of girth $g$ and order $(d-1)^{(1+o(1))\frac{g}{2}}$, then $\delta = d  \geq n^{\frac{2-o(1)}{g}}$, implying that we can apply Theorem~\ref{Thm: Cop Number Upper Bound} to bound the cop number of such graphs for fixed $g$ and large $d$. This would give a bound of the form 
$$
c(G) \leq \frac{(6+o(1))g(d-1)^{(1+o(1))\frac{g}{2}}}{2(d-1)^{\lfloor \frac{g+1}{4}\rfloor}}  = (d-1)^{(1+o(1))\frac{g}{4}}.
$$
Thus, proving that there exists $(d,g)$-cages of order $(d-1)^{(1+o(1))\frac{g}{2}}$ is sufficient to prove Conjecture~\ref{Conj: Large Girth Small Cop Number}.

\section{Proof of Theorem~\ref{Thm: Cop Number Upper Bound}}

Our proof is in part probabilistic and proceeds similarly to arguments that appear in \cite{lu2012meyniel, scott2011bound, bradshaw2023cop}. For readers unfamiliar with the probabilistic method  we recommend \cite{alon2016probabilistic} as a reference. Additionally, we adopt notation specific to cops and robbers from \cite{bonato2011game}. For a graph $G=(V,E)$ and a vertex $v\in V$ we define $B_r(v) := \{u \in V: \dist(u,v) \leq r\}$.

\begin{lemma}[\cite{hosseini2020cops}]\label{Lemma:the lemma}
    Let $G = (V,E)$ be a graph with girth $g$ and let $t = \lfloor \frac{g+1}{4} \rfloor$. If $u \in V$ and the robber is not in $B_{2t-2}(u)$, then $2$ cops can prevent the robber from entering $B_t(u)$.
\end{lemma}

We are now prepared to prove the main theorem.

\begin{proof}[Proof of Theorem~\ref{Thm: Cop Number Upper Bound}]
Let $G = (V,E)$ be a graph with girth $g$ and minimum degree $\delta\geq 2$.
Let $t = \lfloor \frac{g+1}{4} \rfloor$.
We let $C \subset V$ be a random subset satisfying for all $v\in V$
$$
\mathbb{P}(v\in C) = p = \frac{2\log n}{(\delta-1)^{t}}.
$$
We will show that with positive probability for every vertex $u\in V$, there exists a $v\in C$, such that $u \in B_{t}(v)$. Next, using Lemma~\ref{Lemma:the lemma} we will argue that this implies that if $2$ cop beings on each vertex of $C$, the cops will capture the robber.

For a vertex $u\in V$ let $E_u$ be the event that 
$$
u \notin \bigcup_{v\in C} B_t(v).
$$
Let $u\in V$ be fixed.
As $G$ has girth $g$, $G[B_{g-1}(v)]$ is a tree, and as $G$ has minimum degree $\delta\geq 2$ it must be the case that $|B_t(v)|> (\delta-1)^t$ given $t < g$. Thus, the probability that no vertex within distance $t$ of $u$ is in $C$, that is the probability of $E_u$, satisfies 
\begin{eqnarray*}
    \mathbb{P}(E_u) & < & (1-p)^{(\delta-1)^t} \\
    & \leq & \exp (-p(\delta-1)^t) \\
    & = & \exp (-2\log n) \\
    & = & \frac{1}{n^2}.
\end{eqnarray*}
Observe that $u\in V$ was chosen without loss of generality.
Applying the union bound over all choice of $u$, 
$$
\mathbb{P}(\exists u\in V, E_u) < \frac{n}{n^2} = o(1)
$$
for all $n \geq 1$. So with positive probability for all vertices $u\in V$, there exists a $v\in C$, such that $u \in B_{t}(v)$.
Observe that $|C|$ is a binomial random variable. Hence, we can applying a Chernoff bound to see that
\begin{eqnarray*}
\mathbb{P}(|C| \geq \mathbb{E}(|C|)+k) \leq \exp(- \frac{k^2}{2(\mathbb{E}(|C|)+ \frac{k}{3})}).
\end{eqnarray*}
Letting $k = \frac{1}{2}\mathbb{E}(|C|)$, we note this implies asymptotically almost surly $|C| < \frac{3}{2}\mathbb{E}(|C|)$. As both the probability that no event $E_u$ occurs and the probability that $|C| < \frac{3}{2}\mathbb{E}(|C|)$ tend to $1$ as $n$ grows, we conclude that for large enough $n$ there is as set $C$ satisfying both requirements. Notice that 
$$
\frac{3}{2}\mathbb{E}(|C|) = \frac{3n\log(n)}{(\delta-1)^{\lfloor\frac{g+1}{4}\rfloor}}.
$$
Now suppose that the cop player beings the game with $2$ cops at each vertex of such a set $C$ which satisfies both no event $E_u$ occur and $|C| < \frac{3}{2}\mathbb{E}(|C|)$. Also suppose that one additional cop we call $x$ begins the game, at some fixed but arbitrary vertex in the graph. Let $r$ denote the current location of the robber, and update this as the game progresses. 

For a vertex $v\in V$ if $v \notin B_{2t-2}(r)$, then Lemma~\ref{Lemma:the lemma} implies that the two cops beginning on a vertex $v$, can prevent the robber from entering $B_t(v)$. For all vertices $v \in C$ such that $v \notin B_{2t-2}(r)$ suppose that the cops beginning at $v$ do this. For each vertex $v \in C$ such that $v \in B_{2t-2}(r)$ the cops at $v$ will pass their turn as long as $v \in B_{2t-2}(r)$. If the robber ever moves so that $v \not\in B_{2t-2}(r)$, then the two cops at $v$ begin preventing the robber from entering $B_t(v)$ using the strategy from Lemma~\ref{Lemma:the lemma}.

We note that by our choice of $C$, for all $u\in V$
$$
u \in \bigcup_{v\in C} B_t(v).
$$
Hence, if the cops at each vertex $v\in C$ are able to prevent the robber from entering $B_t(v)$, then the robber has no vertex to occupy and is therefore captured. 

Let $\{v_1,\dots, v_{|C|}\}$ be some ordering of the vertices in $C$. Suppose that $i$ is the least integer such that $v_i$ has been in $B_{2t-2}(r)$ for the entire game. Let the cop $x$ make their way to $v_i$. If the robber moves during this time such that $v_i\notin B_{2t-2}(r)$, then let $j$ be the new smallest integer such that $v_j$ has been in $B_{2t-2}(r)$ for the entire game and begin again. 
Suppose then that the robber moves so that $v_i \in B_{2t-2}(r)$ at all times. Then $x$ will reach $v_i$. 
As $G$ is a graph of girth $g$ and as $t = \lfloor \frac{g+1}{4} \rfloor$, we note that the cop $x$ can take a shortest path to the robber, forcing the robber to move so that $v_i\notin B_{2t-2}(r)$ or be captured. By doing this, the cop $x$ can increase the index of the smallest $i$ such that $v_i$ has been in $B_{2t-2}(r)$ for the entire game, until for all vertices $v\in C$, the cops who began at $v$ are preventing the robber from entering $B_t(v)$. Thereby capturing the robber.

Therefore, 
$$
2|C| + 1 \leq \frac{6n\log(n)}{(\delta-1)^{\lfloor\frac{g+1}{4}\rfloor}}
$$
cops suffice to capture the robber. This completes the proof.
\end{proof}

We note that our assumption that the graphs in question be sufficiently large is required. Supposing there exists $d$-regular graphs $G$ of order $(d-1)^{(1+o(1))\frac{g}{2}}$, the coefficient of $\frac{1}{4}$ in the exponent of Theorem~\ref{Thm: Cop Number Upper Bound} cannot be improved for small graphs without contradicting the lower bound  $c(G) \geq \frac{1}{g} (\delta - 1)^{\lfloor \frac{g-1}{4}\rfloor}$ from \cite{bradshaw2023cop}.

\section{Conclusion}

In this paper we have extended an upper bound for the cop number given by Frankl for $d$-regular of high girth appearing in \cite{frankl1987cops} to all graphs girth at least $7$. In doing so we have shown that such graphs have cop number $O(n\log(n)(\delta-1)^{-\lfloor \frac{g+1}{4} \rfloor})$ where the asymptotics here are in $n$. This result implies that for all $\alpha \in (0,1)$ there exists a sufficiently large $0< \epsilon< 1$ such that if $G$ has $\delta \geq n^{\epsilon}$ and $G-S$, for $S\subseteq V(G)$, has no cycle of length $7$ or less where $|S| = O(n^{1-\alpha})$, then $c(G) = O(n^{1-\alpha})$. We conjecture that this upper bound is not best possible. Additionally, our work adds further evidence to a conjecture of Bradshaw, Hosseini, Mohar, and Stacho from \cite{bradshaw2023cop} by showing that their conjecture is implied by a conjecture related to the Moore bound without the need for Meyniel's conjecture to be true.

\section*{Acknowledgement}

We would also like to acknowledge the support of the Natural Sciences and Engineering Research Council of Canada (NSERC) for support through the Canadian Graduate Scholarship -- Master's program. Additionally, we would like to thank the reviewers for their thorough work, insightful suggestions, and encouraging words.
In particular we would like to 
This paper benefited greatly from their efforts.

\bibliographystyle{plain}
\bibliography{CopsRobbers}

\end{document}